\documentclass[10pt]{amsart}
\usepackage{amssymb, amscd, amsthm, epsfig}
\newtheorem{thm}{\bfseries Theorem}
\newtheorem{lem}[thm]{\bfseries Lemma}        

\theoremstyle{definition}
\newtheorem*{ack}{\bfseries Acknowledgement}

\def\F{\mathbb F}

\DeclareMathOperator{\Det}{Det}

\begin{document}


\title[\bf Remarks to Arsovski's proof \hfill {}]
{\bf Remarks to Arsovski's proof of\\
Snevily's conjecture}


\author[ {} \hfill G. Harcos, Gy. K\'arolyi and G. K\'os]{}

\address{\vskip -.1in}


\email{ gharcos@renyi.hu }
\email{ karolyi@cs.elte.hu }
\email{ kosgeza@cs.elte.hu }



\maketitle

\centerline{\em Dedicated to our late friend, Andr\'as G\'acs}

\vskip 1in


\noindent GERGELY HARCOS\footnote{
Supported by European Community Grant ERG 239277 and
OTKA Grants K 72731 and PD 75126.}\hskip.1in
\noindent Alfr\'ed R\'enyi Institute of Mathematics,
Hungarian Academy of Sciences,
POB 127,
Budapest, H--1364 Hungary

\noindent GYULA K\'AROLYI\footnote{
Supported by Bolyai Research Fellowship and
OTKA Grants K 67676 and NK 67867.}\hskip.1in
\noindent Institute of Mathematics,
E\"otv\"os University,
P\'azm\'any P. s\'et\'any 1/C,
Budapest, H--1117 Hungary

\noindent G\'EZA K\'OS
\hskip.1in
\noindent Institute of Mathematics,
E\"otv\"os University,
P\'azm\'any P. s\'et\'any 1/C,
Budapest, H--1117 Hungary, and
Computer and Automation Research Institute, Kende utca 13-17,
Budapest, H--1111 Hungary

\vskip.85in

\setlength{\baselineskip}{15pt}


\noindent\textsc{Abstract.}
Let $G$ denote a finite Abelian group, and $\F$ a field whose
multiplicative group $\F^\times$ contains an element whose order equals
the exponent of $G$. For any pair $A=\{a_1,\ldots,a_k\}$,
$B=\{b_1,\ldots,b_k\}$ of $k$-element subsets of $G$ there exist
homomorphisms $\chi_1,\ldots,\chi_k:G\to\F^\times$ such that neither
of the two matrices $(\chi_i(a_j))$ and $(\chi_i(b_j))$ is singular.
This confirms a conjecture of Feng, Sun, and Xiang. We also give a shortened
proof of Snevily's conjecture.




\section{Introduction}

\noindent
Let $G$ denote a finite Abelian group of order $m$ and exponent $n$.
We say that $G$ is {\em fully representable} over a field $\F$ if its
multiplicative group $\F^\times$ contains an element of order $n$.
This happens if and only if the characteristic of the field $\F$ does not
divide $n$, and $\F$ itself contains the splitting field of the polynomial
$x^n-1$ over its prime field. In this case $\F^\times$ contains a unique
cyclic subgroup $H$ of order $n$ that may be
identified
for every such field of the same characteristic,
and every homomorphism from $G$ to
$\F^\times$ maps into $H$. Such group characters with respect to pointwise
multiplication form the character group $\widehat{G}\cong G$. It follows
from the orthogonality relations
\[\sum_{g\in G}\chi_1(g)\chi_2^{-1}(g)=
\begin{cases}
|G| & \text{if $\chi_1=\chi_2$}\\
\ 0 & \text{otherwise}
\end{cases}\]
that the $m\times m$ matrix $(\chi(g))_{g\in G, \chi\in \widehat{G}}$
is nonsingular. Thus, the characters are linearly independent over $\F$
and form a basis in the vector space of all $G\to \F$ functions over $\F$.
\smallskip

\noindent{\em Remark.}
The independence of the columns of the character table can be interpreted as
follows: For any subset $A=\{a_1,\ldots,a_k\}$ of $G$,
those sets of characters $\chi$ for which the
vectors $\chi(A)=\chi(a_i)_{1\le i\le k}$ are independent over $\F$,
form a rank $k$ matroid $\mathcal{M}_A$ over the ground set $\widehat{G}$.
Here we prove that for any two sets $A,B\subseteq G$ of the same
cardinality, the matroids $\mathcal{M}_A$ and $\mathcal{M}_B$ have a common
basis.

\begin{thm}
Assume that the finite Abelian group $G$ is fully representable over the
field $\F$. For any two subsets $A=\{a_1,\ldots,a_k\}$ and
$B=\{b_1,\ldots,b_k\}$ of $G$ there exist characters $\chi_1,\ldots,\chi_k\in
\widehat{G}$ such that both $\Det(\chi_i(a_j))$ and $\Det(\chi_i(b_j))$
are different from zero.
\label{uj1}
\end{thm}

This confirms a conjecture of Feng, Sun, and Xiang \cite{FSX}. Applying
the natural isomorphism between $G$ and $\widehat{\widehat{G}}$, one obtains
the following dual version.

\begin{thm}
Under the conditions of the previous theorem, let
${\rm X}=\{\chi_1,\ldots,\chi_k\}$ and
$\Psi=\{\psi_1,\ldots,\psi_k\}$ be two
subsets of $\widehat{G}$. Then there exist elements $a_1,\ldots,a_k\in G$
such that both $\Det(\chi_i(a_j))$ and $\Det(\psi_i(a_j))$
are different from zero.
\label{uj2}
\end{thm}

Using the exterior algebra method, Feng, Sun, and Xiang \cite{FSX} pointed
out that a weaker form of Theorem~\ref{uj1} would imply Snevily's conjecture
\cite{snevily}, which after a series of partially successful attempts
\cite{A,DKSS,GW}, see also \cite{Sun,TV}, was recently proved by Arsovski
\cite{Ars}. Thus, one obtains the following affirmative answer for Snevily's
problem.

\begin{thm}
Let $G$ be an Abelian group of odd order.
For any two subsets $A=\{a_1,\ldots,a_k\}$ and
$B=\{b_1,\ldots,b_k\}$ of $G$ there exists a permutation $\pi\in S_k$ such
that the elements $a_1+b_{\pi(1)},\ldots,a_k+b_{\pi(k)}$ are pairwise
different.
\label{snev}
\end{thm}

The proof of Theorem~\ref{uj1}, at least for finite fields $\F$ of
characteristic 2, is implicit in Arsovski's paper. Here we present a variant
of his argument with considerable simplifications, which completely settles
the conjecture of Feng, Sun, and Xiang.

\section{The Proofs}

\noindent
Assume that, for a given finite Abelian group $G$, the statement of Theorem
\ref{uj1} fails for a certain field $\F$ of characteristic $c$, then it also
fails for every field of characteristic $c$ over which $G$ is fully
representable. In particular, it fails for the
purely transcendental extension $\F'=\F(t_1,\ldots,t_m)$.
Accordingly, let $A=\{a_1,\ldots,a_k\}$ and
$B=\{b_1,\ldots,b_k\}$ be two subsets of $G$ such that
$$\Det(\chi_i(a_j))\Det(\chi_i(b_j))=0$$
holds for every $k$-tuple of characters $\chi_1,\ldots,\chi_k\in \widehat{G}$.
Write $\widehat{G}=\{\chi_1,\ldots,\chi_m\}$.
Let $\varphi$ denote an arbitrary function from $G$ to $\F'$; it can be
uniquely expressed as $\varphi=\sum_{u=1}^m\lambda_u\chi_u$ with
Fourier-coefficients $\lambda_u\in \F'$. Consider the $k\times m$ matrices
$M=(m_{iu})$ and $N=(n_{iu})$ with $m_{iu}=\lambda_u\chi_u(a_i)$,
resp. $n_{iu}=\chi_u(b_i)$. In view of the Cauchy--Binet formula
and the multilinearity of the determinant, for the
$k\times k$ matrix $L$ with $(i,j)$ entry $\varphi(a_i+b_j)$ we obtain
\begin{align*}
\Det L&\ =\ \Det (\sum_{u=1}^m\lambda_u\chi_u(a_i+b_j))\ =\
\Det (\sum_{u=1}^m\lambda_u\chi_u(a_i)\chi_u(b_j))\\
&\ =\ \Det(MN^\top)\ =\
\sum_{1\le u_1<\dots<u_k\le m}\Det(m_{iu_j})\Det(n_{iu_j})\\
&\ =\ \sum_{1\le u_1<\dots<u_k\le m}(\lambda_{u_1}\cdots\lambda_{u_k})
\Det(\chi_{u_i}(a_j))\Det(\chi_{u_i}(b_j))\\
&\ =\ 0.
\end{align*}

Enumerate the elements of $G$ as $g_1,\ldots,g_m$, and apply the above formula
for the function $\varphi$ that maps each $g_i$ to the corresponding $t_i$.
Then $\Det L$ is the alternating sum of $k!$ monomial terms in
$t_1,\ldots,t_m$, each of degree $k$. Because of the algebraic independence of
the elements $t_i$, $\Det L$ can only vanish if each monomial term cancels
out, either because it appears with both $+$ and $-$ signs, or because it
appears at least $c$ times with the same sign. Anyway, for any permutation
$\pi\in S_k$ there exists a permutation $\sigma\ne \pi\in S_k$ such that
the elements $a_1+b_{\sigma(1)},\ldots,a_k+b_{\sigma(k)}$, in some order,
coincide with the elements $a_1+b_{\pi(1)},\ldots,a_k+b_{\pi(k)}$.
According to the following simple combinatorial lemma inherent in \cite{Ars},
this is impossible.

\begin{lem}
Let $A=\{a_1,\ldots,a_k\}$ and $B=\{b_1,\ldots,b_k\}$ be subsets of an
arbitrary Abelian group $G$. There exists a permutation $\pi\in S_k$ such
that for any permutation $\sigma\ne \pi\in S_k$, the multisets
$\{a_1+b_{\pi(1)},\ldots,a_k+b_{\pi(k)}\}$ and
$\{a_1+b_{\sigma(1)},\ldots,a_k+b_{\sigma(k)}\}$ are different.
\label{lemma}
\end{lem}

\begin{proof}
Fix the positive integer $k$, and assume that the lemma has already been
verified for smaller values of $k$. Write $a_1+b_1=g$, and
consider the set $I$ of all indices $i$ for which there exists an index $j$
with $a_i+b_j=g$. We may assume that $I=\{1,\ldots, \ell\}$. In
the case $\ell=k$
there is a unique permutation $\pi$ with $a_1+b_{\pi(1)}=\dots=a_k+b_{\pi(k)}
=g$. If $1\le \ell<k$, then fix the first $\ell$ values of $\pi$ by
$a_1+b_{\pi(1)}=\dots=a_\ell+b_{\pi(\ell)}=g$, and
apply the induction hypothesis for the multisets
$A'=\{a_{\ell+1},\ldots,a_k\}$, $B'=\{b_i\mid i\ne \pi(1),\ldots,\pi(\ell)\}$
to extend it to a permutation $\pi\in S_k$.
\end{proof}

This contradiction proves Theorem~\ref{uj1}.
Theorem~\ref{snev} follows by the sophisticated argument of \cite{FSX} or
by the elegant reasoning of Arsovski \cite{Ars}.
In retrospect, the proof only relies on the identity
(valid in characteristic 2)
\[\Det(\varphi(a_i+b_j))
\ =\ \sum_{1\leq u_1<\dots<u_k\leq m}\ \sum_{\pi\in S_k}
\Det(\lambda_{u_j}\chi_{u_j}(a_i+b_{\pi(i)}))\]
and the existence of $\varphi=\sum_{u=1}^m\lambda_u\chi_u:G\to\F'$, guaranteed by Lemma~\ref{lemma}, for
which the left hand side is nonzero. Indeed, if Theorem~\ref{snev} fails then each
determinant on the right hand side is zero because the underlying matrix has two equal rows. The indentity can be proved directly using the multilinearity of the determinant and the multiplicativity
of the characters $\chi_u$:
\begin{align*}
\Det(\varphi(a_i+b_j))
&\ =\ \sum_{1\leq u_1,\dots,u_k\leq m}\ \Det(\lambda_{u_i}\chi_{u_i}(a_i+b_j))\\
&\ =\ \sum_{\substack{1\leq u_1,\dots,u_k\leq m\\\text{distinct}}}\ \Det(\lambda_{u_i}\chi_{u_i}(a_i+b_j))\\
&\ =\ \sum_{\substack{1\leq u_1,\dots,u_k\leq m\\\text{distinct}}}\ \sum_{\pi\in S_k}\prod_{i=1}^k(\lambda_{u_i}\chi_{u_i}(a_i+b_{\pi(i)}))\\
&\ =\ \sum_{1\leq u_1<\dots<u_k\leq m}\ \sum_{\pi\in S_k}
\Det(\lambda_{u_j}\chi_{u_j}(a_i+b_{\pi(i)})).
\end{align*}
It would be very interesting to find a purely combinatorial proof.

\begin{ack}
We are grateful to Andr\'as Bir\'o, Andr\'as Frank, P\'eter P\'al P\'alfy and
Csaba Szab\'o for discussions over this topic.
\end{ack}

\end{document}